\documentclass[11pt,thmsa]{article}%
\usepackage{amsmath}
\usepackage{amsfonts}
\usepackage{amssymb}
\usepackage{graphicx}%
\newtheorem{theorem}{Theorem}[section]

\newtheorem{lemma}[theorem]{Lemma}

\newtheorem{proposition}[theorem]{Proposition}

\newenvironment{proof}[1][Proof]{\noindent\textbf{#1.} }{\ \rule{0.5em}{0.5em}}

\begin{document}
\title{Geodesic orbit Finsler metrics on  Euclidean spaces}
\author{Ming Xu${}^1$\thanks{Supported by NSFC (11821101, 11771331) and Beijing Natural Science Foundation
(No. 1182006).},
Shaoqiang Deng${}^2$\thanks{Supported by NSFC (11671212,  51535008) and the Fundamental Research Funds for the Central Universities. }
 and Zaili Yan${}^3$\thanks{Corresponding Author. Supported by  NSFC (11626134, 11701300) and K.C. Wong Magna Fund in Ningbo University. }
\\ \\
${}^1$School of Mathematical Sciences\\
Capital Normal University\\
Beijing 100048, P. R. China\\
Email:mgmgmgxu@163.com
\\ \\
${}^2$School of Mathematical Sciences and LPMC\\
Nankai University\\
Tianjin 300071, P. R. China\\
Email:dengsq@nankai.edu.cn
\\ \\
${}^3$Department of Mathematics\\
Ningbo University\\
Ningbo, Zhejiang 315211, P. R. China\\
Email:yanzaili@nbu.edu.cn
}
\date{}
\maketitle
\begin{abstract}
A Finsler space $(M,F)$  is called a geodesic orbit space if any geodesic of constant speed is the orbit of a one-parameter subgroup of isometries of $(M, F)$. In this paper, we study Finsler metrics on Euclidean spaces which are geodesic orbit metrics.  We will show
that, in this case $(M, F)$ is a fiber bundle over a symmetric Finsler space $M_1$ of non-compact type such that each fiber $M_2$ is a totally geodesic nilmanifold with a step-size at most 2, and the projection $\pi:M\rightarrow M_1$ is a Finslerian submersion.
Furthermore,  when $M_1$ has no Hermitian symmetric factors, the fiber bundle description for $M$ can be strengthened to $M=M_1\times M_2$ as coset spaces, such that each product factor is totally geodesic in $(M,F)$ and is a geodesic orbit Finsler space itself. Finally,  we use the techniques in this paper to discuss the interaction between the geodesic orbit spaces and the negative  (non-positive) curved conditions, and provide new proofs for some of our previous results.


\end{abstract}

\section{Introduction}
A homogeneous Riemannian or Finsler manifold is called
a {\it geodesic orbit space}, if any geodesic is the orbit
of a one-parameter subgroup of isometries.
The notion of a geodesic orbit space was introduced in Riemannian geometry by O. Kowalski and L. Vanhecke in 1991 \cite{KV1991}, which is a generalization of the naturally reductive homogeneity. Geodesic orbit Riemannian manifolds
have been studied rather extensitively; see for example \cite{AA2007, AN2009, AV1999, BN2018, DKN, Go1996, GN2018}. Meanwhile, Geodesic orbit Finsler spaces and
their subclasses, for example, normal homogeneous and
$\delta$-homogeneous Finsler spaces,
were also studied in recent years; See
\cite{Xu2018-preprint, XD2017, XD2017-2, XZ2018, YD2014}.

In this paper, we study  geodesic orbit Finsler spaces $(M,F)$ in the case
 that $M$ is diffeomorphic to a Euclidean space. Our main theorem is the following

\begin{theorem}\label{main-thm-1}
Let $(M,F)$ be a geodesic orbit Finsler space which is diffeomorphic to a Euclidean space. Then $M$ is the total
space of a fiber bundle $\pi:M\rightarrow M_1$ such that $\pi$
is a Finslerian submersion, $M_1$ is a symmetric  Finsler space of non-compact type, and each fiber  is a geodesic orbit nilmanifold with  step-size at
most $2$. Furthermore, each fibre is  totally geodesic
in $(M,F)$.
\end{theorem}

Theorem \ref{main-thm-1} is a generalization of the description for geodesic orbit Riemannian metrics on Euclidean spaces in \cite{GN2018}. The proof of Theorem \ref{main-thm-1} applies a different theme, mainly originated from the submersion and
totally geodesic techniques. The interaction between the geodesic orbit condition and a Levi decomposition provides the most crucial key lemmas (see Lemma \ref{key-lemma-1}, Lemma \ref{key-lemma-2} and Lemma \ref{key-lemma-3}).

Now we give some remarks on some  special cases of $(M,F)$ in Theorem \ref{main-thm-1}.

First, the Finsler space $M_1$ in Theorem \ref{main-thm-1} is a symmetric space of non-compact type.
If $M_1$ has no symmetric Hermitian factors, then as a coset space, $M$ can be presented
as the product $M=M_1\times M_2$, in which each product
factor is totally geodesic in $(M,F)$ and geodesic orbit itself
(see Theorem \ref{main-thm-2} for the precise statement).

If $M_1$ is a single point, then $(M,F)$ is
a geodesic orbit Finsler nilmanifold.
Theorem \ref{2-step-nil-go-thm} in Section \ref{section 4}, which asserts that the step-size is at most $2$, was first given as Theorem 5.2 in \cite{YD2014}. Here we apply a different approach, which fixes the gap in the proof in \cite{YD2014}.

If the fiber $M_2$ in Theorem is a single  point, then $(M,F)$ is
a symmetric Finsler space of non-compact type,
which has non-positive flag curvature and negative Ricci scalar.
Furthermore, if $M_1$ is of rank one, then $(M,F)$ has negative flag curvature.
The inverse statements are also true, and the rigidity of the metric $F$ can be implied by the interaction between the geodesic orbit and curved conditions. See Theorem \ref{main-cor-1}
and Theorem \ref{main-cor-2}, which first appear in \cite{XD2017}.
We use the techniques in this paper to give new proofs for the above two theorems.

This paper is organized as follows. In Section 2, we summarize some fundamental facts on general Finsler geometry and homogeneous Finsler geometry which will be used in later discussion. In Section 3, we review the definition of a geodesic
orbit Finsler space and give some fundamental results on such spaces. In Section 4, we prove the nilradical of $G$ has a step-size at most 2 when $G/H$
admits a $G$-geodesic orbit Finsler metric, as the corollary, we prove the step-size for a geodesic orbit Finsler nilmanifold is at most 2. In Section 5, we discuss the interaction between the geodesic orbit condition and a Levi decomposition and prove several key lemmas. In Section 6, we prove Theorem \ref{main-thm-1} and provide further discussions for the special cases mentioned above.

{\bf Acknowledgement}. The authors would like to express their sincere thanks to Professors Yuri G. Nikonorov and C. Gordon for helpful discussions and
valuable suggestions.

\section{Preliminaries}
\subsection{Finsler metric and Minkowski norm}

In this section, we give some fundamental facts about Finsler spaces. Throughout this paper, manifolds are always assumed to be connected and smooth.

A {\it Finsler metric} on an $n$-dimensional manifold $M$
is a continuous function $F:TM\rightarrow [0,+\infty)$, which satisfies
the following conditions \cite{CS2005}:
\begin{description}
\item{\rm (1)}  The restriction of $F$ to the slit tangent bundle $TM\backslash
0$ is a positive smooth function.
\item{\rm (2)}  For any $\lambda\geq 0$,
$F(x,\lambda y)=\lambda F(x,y)$.
\item{\rm (3)} For any {\it standard local coordinates} $x=(x^i)\in M$
and $y=y^j\partial_{x^j}\in T_xM$, the Hessian matrix
$$(g_{ij}(x,y))=\left(\frac12[F^2(x,y)]_{y^iy^j}\right)$$
is positive definite.
\end{description}

We will call $(M,F)$ a {\it Finsler manifold} or a {\it Finsler space}.
The restriction of $F$ to a tangent space $T_xM$, $x\in M$,  is called a {\it Minkowski norm}. More generally, a
Minkowski norm can be defined on any real linear space; See \cite{BCS} for details.

Given  a  nonzero vector $y$  in $T_xM$, the Hessian matrix $(g_{ij}(x,y))$ defines an inner product $\langle \cdot,\cdot\rangle_y^F$,
such that for any $u=u^i\partial_{x^i}$
and $v=v^j\partial_{x^j}$ in $T_xM$,
$$\langle u,v\rangle_y^F=\frac{1}{2}
\frac{\partial^2}{\partial r\partial s}F^2(y+ru+sv)|_{r=s=0}
=u^iv^j g_{ij}(x,y).$$
For a Minkowski norm on a real vector space, the Hessian matrix $(g_{ij}(y))$ defines the inner product $\langle\cdot,\cdot\rangle_y^F$ similarly.

\subsection{Homogeneous Finsler space}
A Finsler space $(M,F)$
is said to be {\it homogeneous} if its connected isometry group $I_0(M,F)$ acts transitively. For any
closed subgroup $G\subset I_0(M,F)$ which acts transitively
on $M$, we can present $M$ as $M=G/H$, where $H$ is the isotropy subgroup  at $o=eH$, and  is a compact subgroup of $G$. Denote $\mathrm{Lie}(H)=\mathfrak{h}$.
Then an $\mathrm{Ad}(H)$-invariant linear decomposition $\mathfrak{g}=\mathfrak{h}+\mathfrak{m}$ is called a {\it reductive decomposition}.
The subspace
$\mathfrak{m}$ can be $H$-equivalently identified with the tangent space
$T_o(G/H)$, and a $G$-invariant metric $F$ on $G/H$ is completely determined by its restriction in $T_o(G/H)$, i.e.,  an $\mathrm{Ad}(H)$-invariant Minkowski norm on $\mathfrak{m}$
\cite{De2012}.

\subsection{Geodesic and geodesic spray}
On a Finsler space $(M,F)$, a smooth curve $c=c(t)$ is called a {\it
geodesic} if the curve  $(c(t),\dot{c}(t))$ on $TM$
is the integration curve of the {\it geodesic spray} vector field $\mathbf{G}$ on $TM\backslash 0$.  With respect to
a standard local coordinate system  $x=(x^i)\in M$ and $y=y^j\partial_{x^j}\in T_xM$, the geodesic spray can be expressed as
$$\mathbf{G}=y^i\partial_{x^i}-2\mathbf{G}^i\partial_{y^i},$$
where $\mathbf{G}^{i}=\frac{1}{4}g^{il}([F^2]_{x^ky^l}y^k-[F^2]_{x^l})$.
 The equations defining the geodesic $c=c(t)$ are
$$\ddot{c}^i(t)+\mathbf{G}^i(c(t),\dot{c}(t))=0,\quad\forall i.$$

Notice that the notion of geodesic here implies that $F(\dot{c}(t))\equiv\mathrm{const}>0$ \cite{BCS}, that is,
in this paper, we will only consider  geodesics of constant speed.

For a homogeneous Finsler space $(G/H,F)$, the geodesic spray
$\mathbf{G}(x,y)$ is completely  determined by its value at $o=eH$,
as indicated  by the following proposition.

\begin{proposition}\label{homo-geod-spray-thm}
Let $(G/H,F)$ be a homogeneous Finsler space associated
with a reductive decomposition $\mathfrak{g}=\mathfrak{h}+\mathfrak{m}$, and $\{u_1,\ldots,u_n\}$ be a basis of $\mathfrak{m}$ with
$[u_i,u_j]_\mathfrak{m}=c_{ij}^k u_k$. Then for $y=y^i u_i\in T_o(G/H)=\mathfrak{m}$,  we have
\begin{equation}\label{010}
\mathbf{G}(o,y)=\tilde{y}-g^{il}c_{lj}^k [F^2]_{y^k}y^j\partial_{y^i}.
\end{equation}
\end{proposition}

The vector $\tilde{y}\in T_{(o,y)}(T(G/H))$ in Proposition \ref{homo-geod-spray-thm} is defined as the following. Any $y\in\mathfrak{m}\subset\mathfrak{g}$
defines a Killing vector field $Y$ of $(G/H,F)$, and $Y$ induces
a vector field $\tilde{Y}$ on $T(G/H)$. Then
$\tilde{y}$ is the value of $\tilde{Y}$ at $(o,y)\in T(G/H)$.

Proposition \ref{homo-geod-spray-thm} is a reformulation
of Theorem 3.1 in \cite{XD2014}. We omit the proof here.

The geodesic spray can also be determined by $\mathbf{G}(o,y)$ by the {\it spray vector field}
$\eta:\mathfrak{m}\backslash\{0\}\rightarrow\mathfrak{m}$
defined in \cite{Hu2015}. Recall that, given $y\in\mathfrak{m}\backslash\{0\}$, $\eta(y)$ is determined by
$$\langle\eta(y),v\rangle_y^F=\langle y,[v,y]_\mathfrak{m}\rangle_y^F,\quad\forall v\in\mathfrak{m}.$$
With respect to the basis $\{u_1,\ldots,u_n\}$ of $\mathfrak{m}$, we have
$$\eta(y)=\eta^i u_i=g^{il} c_{lj}^k g_{km}y^my^j u_i=
g^{il}c_{lj}^k [F^2]_{y^k}y^j u_i,$$
Thus (\ref{010}) can also be expressed as
$$\mathbf{G}(o,y)=\tilde{y}-\eta^i \partial_{y^i}.$$

\subsection{Totally geodesic submanifolds}
An $n$-dimensional submanifold $N$ of an $m$-dimensional Finsler space $(M,F)$ can be endowed with
the induced submanifold Finsler metric $F'=F|_{N}$. We call $N$ a {\it totally geodesic} submanifold if any geodesic of
$(N,F')$ is also a geodesic of $(M,F)$. By Theorem 3.1 in \cite{XD2014} (or Proposition \ref{homo-geod-spray-thm}), an equivalent condition  for $N$ to be totally geodesic in $(M,F)$ can be given by local tangent frames, i.e.,  around each point $x\in N$ we can find local tangent frame $X_i$ and the corresponding linear coordinates $y=y^iX_i\in TM$, such that
$N$ is spanned by $X_i$ with $i\leq n=\dim N$. Moreover,  at $x$ the geodesic spray
$\mathbf{G}(x,y)=y^i\tilde{X}_i-2\mathrm{G}^i\partial_{y^i}$
satisfies
\begin{equation}\label{012}
\mathrm{G}^i(x,y)=0 \mbox{ when }y\in T_xN\mbox{ and }i>n.
\end{equation}

In the homogeneous case, the above observation gives the following criterion of totally geodesic homogeneous subspaces from Proposition \ref{homo-geod-spray-thm}.

\begin{lemma}
\label{criterion for totally geodesic homogeneous subspace}
Let $(G/H,F)$ be a homogeneous Finsler space with a reductive
decomposition $\mathfrak{g}=\mathfrak{h}+\mathfrak{m}$. Let $G'$ be a closed subgroup of $G$ whose Lie algebra $\mathfrak{g}'$ satisfies the condition
$\mathfrak{g}'=\mathfrak{g}'\cap\mathfrak{h}+\mathfrak{g}'\cap\mathfrak{m}$.
Then $G'/G'\cap H$ is totally geodesic if and only if the spray vector
field $\eta(\cdot)$ of $G/H$ satisfies
\begin{equation}\label{013}
\eta(y)\subset\mathfrak{g}'\cap\mathfrak{m},\quad \forall y\in\mathfrak{g}'\cap\mathfrak{m}\backslash \{0\}.
\end{equation}
\end{lemma}

\begin{proof}
We first prove that $G'/G'\cap H$ is totally geodesic when  (\ref{013}) is satisfied.

Notice that  $\mathfrak{g}'=\mathfrak{g}'\cap\mathfrak{h}
+\mathfrak{g}'\cap\mathfrak{m}$ is a reductive decomposition for the coset space $G'/G'\cap H$. It is easily seen that there is  a basis $\{u_1,\ldots, u_m\}$ of $\mathfrak{m}$ such that the elements $u_i$',  $i\leq n\leq m$,  span $\mathfrak{g}'\cap\mathfrak{m}$. This basis defines a local
tangent frame of $G/H$ at $o\in G'/G'\cap H$. Then  the assumption
(\ref{013}) implies that (\ref{012}) is valid at $o$.

Since the replacement of $o$ with  $g\in G'$ is just an $\mathrm{Ad}(g)$-change for the spray vector field $\eta(\cdot)$,  (\ref{013}) is still satisfied. This
implies that (\ref{012}) is valid at any point $x\in G'/G'\cap H$.
So $G'/G'\cap H$ is totally geodesic.

The same argument as above can also be used to prove (\ref{013})
when $G'/G'\cap H$ is totally geodesic. This completes the proof
of the lemma.
\end{proof}

\subsection{Finslerian submersion}
A {\it linear submersion} $l:(\mathbf{V}_1,F_1)\rightarrow
(\mathbf{V},F_2)$ between two Minkowski spaces is a surjective
linear map such that $l$ maps the $F_1$-unit ball in
$\mathbf{V}_1$ onto the $F_2$-unit ball in $\mathbf{V}_2$.

Given  $v_2\in \mathbf{V}_2$, there exists
a unique $v_1\in l^{-1}(v_2)$ such that
$$F_1(v_1)=\mathrm{inf}\{F_1(v)| l(v)=v_2\}.$$
We call  $v_1$ the {\it horizonal lift} of $v_2$.

The smooth map $f:(M_1,F_1)\rightarrow (M_2,F_2)$
between two Finsler spaces is called a {\it Finslerian submersion} or simply a {\it submersion}, if for any $x_1\in M_1$, the tangent map
$f_*:(T_{x_1}M_1,F_1)\rightarrow(T_{x_2}M_2,F_2)$, where $x_2=f(x_1)$,  is a linear submersion \cite{AD2001}.

Given a smooth map $f:M_1\rightarrow M_2$ between
two manifolds, and a Finsler metric $F_1$ on $M_1$,
it is natural to ask if there exists a Finsler metric $F_2$
on $M_2$ such that $f$ is Finslerian submersion.
If such $F_2$ exists, we call it the {\it induced metric
defined by submersion} from $F_1$ and $f$.
The following lemma will be useful.

\begin{lemma}\label{quotient-metric-lemma}
Let $(M,F)$ be a smooth Finsler space, and $G$ a closed subgroup of
$I_0(M,F)$ such that the quotient $G\backslash M$ is a smooth manifold and the quotient map $\pi:M\rightarrow G\backslash M$ has surjective tangent maps everywhere,
then there exists a unique induced metric defined by submersion
from $F_1$ and $\pi$.
\end{lemma}
\begin{proof}For any $x\in M$ and $\bar{x}\in G\backslash M$,
there exists a unique Minkowski norm $F_2$ on $T_{\bar{x}}(G\backslash M)$, such that $\pi_*:(T_xM,F_1)
\rightarrow (T_{\bar{x}}(G\backslash M))$ is a linear submersion. It is clear that $F_2(\bar{x},\cdot)$ depends smoothly on $x$.
 Therefore we only need to proof that $F_2(\bar{x},\cdot)$ does not depend on the representative $x$ for $\bar{x}$. Given $x_1$ and $x_2$ with $\bar{x}=\pi(x_1)=\pi(x_2)$,
since $$\pi_*|_{T_{x_2}M}\circ g_*|_{T_{x_1}M} =\pi_*|_{T_{x_1}M},$$
for any $g\in G$ with $g\cdot x_1=x_2$,  the Minkowski norms at $x_1$ and $x_2$ define the same Minkowski norm at $\bar{x}$. This completes the proof of the lemma.
\end{proof}

A special case of Lemma \ref{quotient-metric-lemma} has been used in \cite{XD2017-2} and \cite{XDHH2017}. See Lemma 3.3 in
\cite{XDHH2017}.

\section{Geodesic orbit Finsler space}
\label{section-geodesic orbit space}
Let  $(M,F)$ be a Finsler space  and $G$ a Lie group acting isometrically  on $(M,F)$. Then we call $(M,F)$
a {\it $G$-geodesic orbit space}, if each geodesic of nonzero constant speed on $(M,F)$ is a {\it homogeneous geodesic} of $G$, i.e., it is  the orbit of
a one-parameter subgroup $\exp tu$ of $G$, where $u\in \mathfrak{g}=\mathrm{Lie}(G)$. In the following, if the group $G$ is not
specified, then it is automatically assumed that $G=I_0(M,F)$.

We remark here that in the above definition we do not assume  the action of $G$ on $M$ to be transitive. However, this can be easily deduced  from  the connectedness of $M$. On the other hand, it is not assumed that the $G$-action is
effective, or the closeness of the
image of $G$ in $I_0(M,F)$ is closed. However, this problem can be settled just by replacing $G$ with the closure $\overline{G'}$ of the image $G'$ of $G$ in $I_0(M,F)$. It is easily seen that  $(M,F)$ is also a $\overline{G'}$-geodesic orbit space.

The following proposition provides several equivalent definitions for a geodesic orbit Finsler space. Notice that by assuming
$(M,F)=(G/H,F)$ to be a homogeneous Finsler space, we mean that $G$ acts effectively on $(M,F)$.

\begin{proposition}\label{prop-1}
Let $(G/H,F)$ be a homogeneous Finsler space, with a reductive decomposition $\mathfrak{g}=\mathfrak{h}+\mathfrak{m}$, and denote $[\cdot,\cdot]_\mathfrak{m}$ the
$\mathfrak{m}$-factor in the bracket operation $[\cdot,\cdot]$. Then the following statements are equivalent:
\begin{description}
\item{\rm (1)} $F$ is a $G$-geodesic orbit metric.
\item{\rm (2)} For any $x\in M$, and any nonzero $y\in T_xM$, there exists a Killing vector
field $X\in\mathfrak{g}$ such that $X(x)=y$ and $x$ is a critical point for the function $f(\cdot)=F(X(\cdot))$.
\item{\rm (3)} For any nonzero vector $u\in\mathfrak{m}$, there exists  $u'\in\mathfrak{h}$
such that
\begin{equation}\label{011}
\langle u,[u+u',\mathfrak{m}]_\mathfrak{m}\rangle^F_u=0.
\end{equation}
\item{\rm (4)} The spray vector field $\eta(\cdot):\mathfrak{m}\backslash \{0\}\rightarrow\mathfrak{m}$
is tangent to the $\mathrm{Ad}(H)$-orbits.
\end{description}
\end{proposition}

For the proof, see \cite{Xu2018-preprint}.

In particular, the equivalence between (1) and (3)
implies the following consequence.

\begin{lemma}\label{lemma-go-centralizer}
Assume $(M,F)=(G/H,F)$ is a $G$-geodesic orbit Finsler space,
with a reductive decomposition $\mathfrak{g}=\mathfrak{h}+\mathfrak{m}$. Then for any
closed connected subgroup $H'\in G$ such that
$\mathfrak{h}'=\mathrm{Lie}(H')$ commutes with $\mathfrak{h}$
and is contained in $\mathfrak{m}$, i.e. $\mathfrak{h}'\subset
\mathfrak{c}_\mathfrak{m}(\mathfrak{h})$, we have the following:
\begin{description}
\item{\rm (1)} $M=G\times H'/H H'$, where the $H'$-factor
in the denominator is diagonally contained in $G\times H'$.
\item{\rm (2)} $F$ is $G\times H'$-invariant.
\end{description}
\end{lemma}
\begin{proof}
Applying Proposition \ref{prop-1} for the $G$-geodesic orbit property of $F$, we can get
$$\langle u,[u,v]_\mathfrak{m}\rangle^F_u=0,\quad\forall u\in\mathfrak{m},v\in\mathfrak{h}'.$$
It implies that the right multiplication of $H'$ induces Killing vector
fields on $(G/H,F)$, and the lemma follows easily.
\end{proof}

Homogeneous totally geodesic subspace and Finslerian submersion
provide important tools  to study geodesic orbit
spaces. Totally geodesic techniques in the Riemannian context
can be naturally generalized to the Finsler situation. For example, we have the following lemma.

\begin{lemma} \label{totally-geodesic-technique-lemma-for-go}
Let $(M,F)$ be a $G$-geodesic orbit Finsler space. For any subset $L$ of isometries fixing $x\in M$, we denote $\mathrm{Fix}_x(L)$ the connected component of the fixed point set
$\mathrm{Fix}(L)$ of the $L$-action on $M$ containing $x$, and $C_G^0(L)$ the identity component of the centralizer $C_G(L)$ of $L$ in $G$. Then the restriction of $F$ to
$\mathrm{Fix}_x(L)=C_G^0(L)\cdot x$ is a
$C_G^0(L)$-geodesic orbit Finsler metric.
\end{lemma}

\begin{proof} Without loss of generality, we can assume
that $L$ is a compact Lie group. Then the submanifold
$\mathrm{Fix}_x(L)=C_G^0(L)\cdot x$ is
totally geodesic.

Consider a geodesic $\gamma$ of $(\mathrm{Fix}_x(L),F|_{\mathrm{Fix}_x(L)})$. Then $\gamma$ is also a geodesic of $(M,F)$. By the geodesic orbit property of $(M,F)$, there exists a Killing vector field $X$ such that $X$ is defined by some vector $v\in\mathfrak{g}$,
and $\gamma$ is an integration curve of $X$.
Since $L$ is a compact Lie group, the average
$$X'=\int_L\mathrm{Ad}(g)X d\mathrm{Vol}/\mathrm{Vol}(L), $$
where $d\mathrm{Vol}$ is a fixed bi-invariant volume form,
is also a Killing vector field of $(M,F)$, which is  defined by
$$v'=\int_L\mathrm{Ad}(g)v d\mathrm{Vol}/\mathrm{Vol}(L).$$
Restricted to the geodesic $\gamma$, $X'$ coincides with $X$, i.e.,  $\gamma$ is an integration curve of $X'$. It is also easy to see that $$v'\in\mathfrak{c}_\mathfrak{g}(L)=\mathrm{Lie}(C^0_G(L)).$$
So
$(\mathrm{Fix}_x(L),F|_{\mathrm{Fix}_x(L)})$ is a $C^0_G(L)$-geodesic orbit Finsler space.
\end{proof}

On the other hand, many submersion techniques for geodesic orbit Riemannian manifolds do not work in the Finsler context.
Fortunately, we still have the following:

\begin{theorem}\label{go-submersion-thm}
Let $(G/H,F)$ be a $G$-geodesic orbit Finsler space such that the $G$-action is effective. Suppose  $H_1$ is
a closed normal subgroup of $G$, and $H_2$ is the maximal normal subgroup of $G$ contained in $ H_1H$. Denote $G'=G/H_2$ and $H'=H_1H/H_2$. Then we have
\begin{description}
\item{\rm (1)} $G'$ acts effectively on $G'/H'$.
\item{\rm (2)} There exists a unique $G'$-invariant metric $F'$ on $G'/H'$ defined by submersion from the metric $F$ and the natural projection $$\pi:G/H\rightarrow G/H_1H=(G/H_2)/(H_1H/H_2)=G'/H'.$$
\item{\rm (3)} $(G'/H',F')$ is a $G'$-geodesic orbit Finsler space.
\end{description}
\end{theorem}

\begin{proof}
 Since $H_1$ is closed and normal in $G$, and $H$ is compact, the product $H_1H$ is a closed subgroup of $G$. The
largest normal subgroup $H_2$ of $G$ contained in $H_1H$ coincides with $\cap_{g\in G}gH_1Hg^{-1}=\cap_{g\in G}H_1 gHg^{-1}$. So $H_2$ contains $H_1$, and $H_2$ is closed in $G$.
The subgroup $H_2$ consists of  all the elements $g\in G$ which
acts as the identity map on $G/H_1H$. Thus  $G'=G/H_2$ acts
effectively on $G/H_1H=(G/H_2)/(H_1H/H_2)=G'/H'$. This proves (1).

To prove (2),  we first note that, since $H_1$ is a normal subgroup of $G$, we have $gH_1H=H_1gH$ for any $g\in G$. Then $G'/H'$ can be identified with $H_1\backslash G/H$,
the orbit space of left $H_1$-actions on $G/H$. Meanwhile,  the quotient map for $H_1\backslash G/H$, $\pi_1:G/H\rightarrow H_1\backslash G/H=G'/H'$ coincides with the projection map $\pi:G/H\rightarrow G/H_1H=G'/H'$.
By Lemma \ref{quotient-metric-lemma}, a unique Finsler metric $F'$ on $H_1\backslash G/H$ is well defined by submersion from $F$ and $\pi'$. It is the metric indicated in (2) of the lemma.
The uniqueness of $F'$ is obvious.

The $G$-action on $G/H$
permutes the fibers of the projection $\pi$. It naturally induces a $G$-action on $G/H_1H=G'/H'$ which becomes effective through $G'$. By the uniqueness of $F'$, the $G'$-actions
are isometries on $(G/H,F')$. This proves (2).

Finally, for any geodesic $\gamma'$ in $(G'/H',F')$, its horizonal lift $\gamma$ in $(G/H,F)$ is a geodesic in $(G/H,F)$ \cite{AD2001}. By the definition of geodesic
orbit spaces, $\gamma$ is the orbit of a one-parameter subgroup, i.e., there exists a nonzero element $u\in \mathfrak{g}$ and $x\in G/H$ such that $\gamma (t)=\exp tu\cdot x$.
Then it is obvious  that $\gamma'$ is the orbit of
$\exp tu'$, where $u'$ is the image of $u$ in
$\mathfrak{g}'=\mathfrak{g}/\mathfrak{h}_2$, where $\mathfrak{g}'$ and $\mathfrak{h}_2$ are the Lie algebras of
$G'$ and $H_2$, respectively. Therefore $(G'/H',F')$
is a $G'$-geodesic orbit space, which proves (3).
\end{proof}

At the end of this section, we remark that Lemma 3.1 in \cite{XD2017} provides an important technique in this paper.
It can be refined to the following lemma, such that we do not
need to assume the effectiveness of the $G$-action, or the closeness for the image of $G$ in the isometry group.

\begin{lemma}\label{GO-Abelian-ideal-in-m-KVFCL}
Let $(G/H,F)$ be a $G$-geodesic orbit Finsler space, and
$\mathfrak{a}$ be an Abelian ideal of $\mathfrak{g}=\mathrm{Lie}(G)$ which has a trivial
intersection with $\mathfrak{h}=\mathrm{Lie}(H)$.
Then each vector in $\mathfrak{a}$ defines a Killing vector field of constant length on $(G/H,F)$.
\end{lemma}

\begin{proof}
We only need to provide a reductive decomposition $\mathfrak{g}=\mathfrak{h}+\mathfrak{m}$,
i.e. $[\mathfrak{h},\mathfrak{m}]\subset\mathfrak{m}$, such that $\mathfrak{a}$ is contained in
$\mathfrak{m}$. Then we can use the argument for Lemma 3.1 in \cite{XD2017} to prove this lemma.

The maximal normal subgroup $G_0$ of $G$ contained in $H$ is the intersection of all $G$-conjugations of $H$, thus it is a closed subgroup of $G$. Denote $\mathfrak{g}_0=\mathrm{Lie}(G_0)$. The coset space $G/H$ can also be presented as $G'/H'=(G/G_0)/(H/G_0)$ such that the metric $F$ is $G'$-invariant and the $G'$-action on $G'/H'$ is effective. Obviously $\mathfrak{a}'=(\mathfrak{a}+\mathfrak{g}_0)/\mathfrak{g}_0$ is an Abelian ideal of $\mathfrak{g}/\mathfrak{g}_0$ and has a zero intersection with
$\mathfrak{h}/\mathfrak{g}_0$. So we may replace $G/H$ and $\mathfrak{a}$ with $G'/H'$ and $\mathfrak{a}'$, then we have
the effectiveness of $G$, i.e. $G$ is a subgroup of $I_0(M,F)$.

Let $\overline{G}$ and $\overline{H}$ be the closure of $G$ and $H$ in $I_0(M,F)$ respectively. The $\mathrm{Ad}(H)$-action
on $\mathrm{Lie}(\overline{G})$ preserves $\mathfrak{h}$, $\mathfrak{a}$ and $\mathfrak{g}$. By continuity, so does $\mathrm{Ad}(\overline{H})$. Since $\overline{H}$ is compact, we can find
an $\mathrm{Ad}(\overline{H})$-invariant inner product on $\mathfrak{g}$ such that
$\mathfrak{h}$ and $\mathfrak{a}$ are orthogonal to each other. With respect to this inner product,
we have an orthogonal reductive decomposition $\mathfrak{g}=\mathfrak{h}+\mathfrak{m}$, such that
$\mathfrak{a}$ is contained in $\mathfrak{m}$.
\end{proof}

\section{Geodesic orbit Finsler nilmanifold}
In this section, we discuss geodesic orbit Finsler metrics on a nilmanifold.

\label{section 4}
We first recall
some fundamental notions in  general theory of Lie groups and Lie algebras.

The {\it radical} $\mathrm{Rad}(G)$
and the {\it nil-radical} $\mathrm{Nil}(G)$
of a connected Lie group $G$ are the connected maximal solvable and nilpotent
subgroups of $G$,  respectively. They are the unique closed subgroups, generated by the maximal solvable and nilpotent
ideals of $\mathfrak{g}$, respectively (see Definition 16.2.1 and Proposition 16.2.2 in \cite{HN2012}).
We denote $\mathfrak{rad}(\mathfrak{g})=\mathrm{Lie}(\mathrm{Rad}(G))$
and
$\mathfrak{nil}(\mathfrak{g})=\mathrm{Lie}(\mathrm{Nil}(G))$,
and call them the {\it radical} and {\it nil-radical} of $\mathfrak{g}$ \cite{HN2012}.

In the case that  $G$ acts
effectively, isometrically and transitively on
a Finsler manifold, we have the following useful lemma.

\begin{lemma}\label{lemma-3-1}
 Let $(M,F)$ be a homogeneous Finsler space, and $G$  a closed connected subgroup in $I_0(M,F)$ which acts transitively on $M$. Then $\mathfrak{n}=\mathfrak{nil}(\mathfrak{g})$ has a zero intersection with
the Lie algebra of any isotropy subgroup of $G$.
\end{lemma}

\begin{proof}
Let $H$ be the compact isotropy subgroup of $G$ at a point $x\in M$. Denote $\mathfrak{n}=\mathfrak{nil}(\mathfrak{g})$.
Obviously $\mathrm{ad}(u):\mathfrak{g}\rightarrow\mathfrak{g}$ is semisimple if
$u\in\mathfrak{h}$, and nilpotent if $u\in\mathfrak{n}$.
So $\mathfrak{h}\cap\mathfrak{n}\subset\mathfrak{c}(\mathfrak{g})$.
Since $G\subset I_0(M,F)$, its action on $M$ is effective. Thus $$\mathfrak{h}\cap\mathfrak{n}=
\mathfrak{h}\cap\mathfrak{c}(\mathfrak{g})=0,$$
which proves the lemma.
\end{proof}

The following result is a generalization
of a theorem of C. Gordon in \cite{Go1996}.

\begin{theorem}\label{theorem-3-2}
 Let $(M,F)$ be a $G$-geodesic orbit Finsler space, where $G$ is a closed connected subgroup of $I_0(M,F)$. Then
the step-size  of the nil-radical $\mathfrak{nil}(\mathfrak{g})$ is at most 2.
\end{theorem}

\begin{proof}
First write the manifold $M$ as $M=G/H$, where  $H$ is a compact subgroup of $G$. Denote $\mathfrak{h}=\mathrm{Lie}(H)$ and
$\mathfrak{n}=\mathfrak{nil}(\mathfrak{g})$.
Then by Lemma \ref{lemma-3-1},
there is a reductive decomposition $\mathfrak{g}=\mathfrak{h}+\mathfrak{m}$ such that
$\mathfrak{n}\subset\mathfrak{m}$.

Assume conversely that the step-size of  $\mathfrak{n}$ is $m>2$, i.e.,  its descending central series
$$C^1(\mathfrak{n})=\mathfrak{n} \mbox{ and }
C^k(\mathfrak{n})=
[\mathfrak{n},C^{k-1}(\mathfrak{n})]\subset C^{k-1}(\mathfrak{n}),$$
satisfies
$C^m(\mathfrak{n})\neq 0$ and $C^{m+1}(\mathfrak{n})=0$.
Then it is easy to see that  each $C^{k}(\mathfrak{n})$ is an ideal of $\mathfrak{g}$ contained in $\mathfrak{m}$.
In particular, $\mathfrak{n}'=C^{m-1}(\mathfrak{n})$ is Abelian with $\dim\mathfrak{n}'>1$. By Lemma \ref{GO-Abelian-ideal-in-m-KVFCL}, any nonzero vector $u$ in
$\mathfrak{n}'$ defines a Killing vector field of constant length on $(M,F)$, that is,
\begin{equation}\label{000}
F(\mathrm{pr}_\mathfrak{m}(\mathrm{Ad}(g)u))
\equiv\mathrm{const}
\end{equation}
for $g\in G$.
Since $\mathfrak{n}'$ is an ideal of $\mathfrak{g}$,
$\mathrm{Ad}(G)u\subset\mathfrak{n}'\subset\mathfrak{m}$.
Setting $g=\exp tv$ with $v\in\mathfrak{n}$ in  (\ref{000}), and taking the differentiation  at $t=0$, we get
\begin{equation}\label{001}
\langle u,[u,v]\rangle^F_u=0,\quad\forall u\in \mathfrak{n}',\,v\in\mathfrak{n}.
\end{equation}

Now we claim that any $v\in\mathfrak{n}$ commutes with $\mathfrak{n}'$. Assume conversely that this is not true. Then
$\mathrm{ad}(v)|_{\mathfrak{n}'}:
\mathfrak{n}'\rightarrow \mathfrak{n}'$
is a nonzero nilpotent linear map. Meanwhile,
on the linear space $\mathrm{End}(\mathfrak{n}')$, we can define two norms as the following. The first one, denoted as $||\cdot||_1$,  is the $l^2$-norm, namely,  with respect to a fixed basis of $\mathfrak{n}'$,
any $A\in\mathrm{End}(\mathfrak{n}')$ corresponds to
a matrix $(a_{ij})$, and $$||A||_1^2={\sum_{i,j}a_{ij}^2}.$$
The second one, denoted as $||\cdot||_2$,  is induced by the norm $F|_{\mathfrak{n}'}$, i.e.,  for  $A\in\mathrm{End}(\mathfrak{n}')$,
$$||A||_2=\sup\{F(Au)|u\in\mathfrak{n}',F(u)=1\}.$$
Then there exists
a basis of $\mathfrak{n}'$, such that the matrix of
$\mathrm{ad}(v)|_{\mathfrak{n}'}\in \mathrm{End}(\mathfrak{n}')$ under this basis is  a nonzero  Jordan form with zero diagonal entries. Now a direct calculation shows that
$$\lim_{t\rightarrow\infty}
||\exp (t\mathrm{ad}(v)|_{\mathfrak{n}'})||_1=+\infty.$$
On the other hand, (\ref{001}) implies that
$$||\exp(t\mathrm{ad}(v)|_{\mathfrak{n}'})||_2=1,\quad\forall t.$$
This is a contradiction,  since the norms  $||\cdot||_1$ and $||\cdot||_2$
 must be equivalent, i.e.,  there exist positive constants $c_1$ and $c_2$, such that $c_1||\cdot||_1\leq ||\cdot||_2\leq c_2||\cdot||_1$.

To summarize, we have proved that
$$C^m(\mathfrak{n})=[\mathfrak{n},\mathfrak{n}']=0,$$
which is a contradiction with the   assumption that the step-size of $\mathfrak{n}$ is $m$.
This completes the proof of the theorem.
\end{proof}

A {\it Finsler nilmanifold} is a Finsler space $(G,F)$, in which $G$ is a nilpotent Lie group and
$F$ is a left invariant Finsler metric. An immediate corollary of Theorem \ref{theorem-3-2} is
the following theorem for the step-size of a geodesic orbit Finsler nilmanifold.

\begin{theorem}\label{2-step-nil-go-thm}
Let $N$ be a nilpotent Lie group which admits a left invariant geodesic orbit Finsler metric. Then the step-size of $N$ is at most 2.
\end{theorem}

\begin{proof}
 Let $F$ be a left invariant geodesic orbit Finsler metric on the nilpotent Lie group $N$,  $G$ the connected isometry group $I_0(N,F)$, and $H$ the isotropy subgroup of $G$ at $e\in N$.
Then $H$ is a compact subgroup, $N$ is a normal subgroup of $G$, and $G$ is the semi-product of $N$ and $H$ \cite{De2011, Wi1982}. Obviously
$\mathfrak{n}=\mathrm{Lie}(N)$ is a nilpotent ideal of $\mathfrak{g}=\mathrm{Lie}(G)$, so we have $\mathfrak{n}\subset\mathfrak{nil}(\mathfrak{g})$. By Lemma
\ref{lemma-3-1}, $\mathfrak{nil}(\mathfrak{g})\cap\mathfrak{h}=0$.
Thus
$\mathfrak{n}=\mathfrak{nil}(\mathfrak{g})$. By Theorem \ref{theorem-3-2}, the step-size of $\mathfrak{n}$ is at most 2.
\end{proof}

\section{Geodesic orbit property of a Levi decomposition}
\label{section-property-levi-decomp}
Let $(M,F)$ be a $G$-geodesic orbit Finsler space, where $G$
is a closed connected subgroup of $I_0(M,F)$. Then $G$ acts on $M$ transitively and effectively,
and the  isotropy subgroup $H$ of $G$ at a fixed $x\in M$ is compact.

Let $B_\mathfrak{g}(\cdot,\cdot)$ be the Killing form of
$\mathfrak{g}$. Since $G$ acts effectively on $M=G/H$,
the restriction $B_\mathfrak{g}|_{\mathfrak{h}\times\mathfrak{h}}$ is non-degenerate. So we have a {\it $B_\mathfrak{g}$-orthogonal reductive decomposition}
$\mathfrak{g}=\mathfrak{h}+\mathfrak{m}$ such that
$\mathfrak{m}$ is the $B_\mathfrak{g}$-orthogonal complement
of $\mathfrak{h}$.

Let $\mathfrak{g}=\mathfrak{r}+\mathfrak{s}$ be the Levi decomposition of $\mathfrak{g}$, i.e., $\mathfrak{r}=\mathfrak{rad}(\mathfrak{g})$ is
the radical, and $\mathfrak{s}=\mathfrak{lev}(\mathfrak{g})$
is a Levi subalgebra of $\mathfrak{g}$. Notice that the nilradical
$\mathfrak{n}=\mathfrak{nil}(\mathfrak{g})$ is contained in
$\mathfrak{r}$. Since $B_\mathfrak{g}(\mathfrak{g},\mathfrak{h})=0$, we have
$\mathfrak{n}\subset\mathfrak{m}$ for the $B_\mathfrak{g}$-orthogonal reductive decomposition.

We further decompose the Levi subalgebra $\mathfrak{s}$ as the Lie algebra direct sum
$\mathfrak{s}=\mathfrak{s}_c\oplus\mathfrak{s}_{nc}$,
where $\mathfrak{s}_c$ is compact semi-simple and
$\mathfrak{s}_{nc}$ is non-compact semi-simple.
We will denote $\mathfrak{k}_{nc}$ a maximal compact subalgebra in $\mathfrak{s}_{nc}$,
$\mathfrak{k}'_{nc}=[\mathfrak{k}_{nc},\mathfrak{k}_{nc}]$,
and $\mathfrak{k}''_{nc}$
the maximal Abelian ideal of $\mathfrak{k}_{nc}$.

The connected subgroups generated by $\mathfrak{k}_{nc}$, $\mathfrak{k}'_{nc}$, $\mathfrak{k}''_{nc}$,
$\mathfrak{n}$, $\mathfrak{r}$, $\mathfrak{s}$,
$\mathfrak{s}_{c}$ and $\mathfrak{s}_{nc}$ are closed subgroups of $G$. They are denoted as
$K_{nc}$, $K'_{nc}$, $K''_{nc}$, $N=\mathrm{Nil}(G)$, $R=\mathrm{Rad}(G)$, $S=\mathrm{Lev}(G)$,
$S_{c}$, $S_{nc}$ respectively.

The following lemma provides a crucial observation.

\begin{lemma}\label{key-lemma-1}
Keeping  all the assumptions and notation in this section, we have
$[\mathfrak{s}_{nc},\mathfrak{r}]=0$.
\end{lemma}

\begin{proof}
We have the following chain of ideals of $\mathfrak{g}$,
$$[\mathfrak{n},\mathfrak{n}]\subset\mathfrak{n}\subset\mathfrak{r}.$$
Each of them generates a closed normal subgroup of $G$.
If we can prove
\begin{eqnarray}
& &[\mathfrak{s}_{nc},[\mathfrak{n},\mathfrak{n}]]=0,
\label{condition-1-in-key-lemma-1}\\
& &[\mathfrak{s}_{nc},\mathfrak{\mathfrak{n}}]
\subset[\mathfrak{n},\mathfrak{n}],\mbox{ and}\label{condition-2-in-key-lemma-1}\\
& &[\mathfrak{s}_{nc},\mathfrak{r}]
\subset\mathfrak{n},\label{condition-3-in-key-lemma-1}
\end{eqnarray}
then the adjoint representation $\mathrm{ad}(\cdot):\mathfrak{r}
\rightarrow\mathfrak{r}$ defines a Lie algebra endomorphism from $\mathfrak{s}_{nc}$ to
a nilpotent Lie algebra. It must be the zero endomorphism since $\mathfrak{s}_{nc}$ is semi-simple,
which proves the lemma.

 Notice that (\ref{condition-3-in-key-lemma-1}) is obvious, since  $\mathfrak{r}\cap[\mathfrak{g},\mathfrak{g}]\subset\mathfrak{n}$.

Let $\mathfrak{g}=\mathfrak{h}+\mathfrak{m}$ be the $B_\mathfrak{g}$-orthogonal reductive decompositon.
By Theorem \ref{theorem-3-2}, $[\mathfrak{n},\mathfrak{n}]$ is an Abelian ideal of $\mathfrak{g}$.
 Since $[\mathfrak{n},\mathfrak{n}]\subset\mathfrak{n}\subset\mathfrak{m}$, by Lemma \ref{GO-Abelian-ideal-in-m-KVFCL},
each nonzero vector $u\in[\mathfrak{n},\mathfrak{n}]$ defines a nonzero Killing vector field of constant length on
$(M,F)$. So we have
$$\langle u,[u,v]\rangle_u^F=0, \quad
\forall  0 \neq u\in[\mathfrak{n},\mathfrak{n}], v\in\mathfrak{s}_{nc}.$$
Assume conversely that $[\mathfrak{s}_{nc},[\mathfrak{n},\mathfrak{n}]]\neq0$.
Then the  adjoint representation defines a Lie algebra endomorphism
from $\mathfrak{s}_{nc}$ into the compact subalgebra $so([\mathfrak{n},\mathfrak{n}],
F|_{[\mathfrak{n},\mathfrak{n}]})$. This is impossible, since each simple ideal in $\mathfrak{s}_{nc}$
is of non-compact type. This proves (\ref{condition-1-in-key-lemma-1}).

Applying Theorem \ref{go-submersion-thm}, we get a submersion from $(G/H,F)$ to $(G'/H',F')=
((G/[N,N])/([N,N]H/[N,N]),F')$ where $N=\mathrm{Nil}(G)$, such that
$F'$ is a $G/[N,N]$-geodesic orbit metric. Since $\mathfrak{n}/[\mathfrak{n},\mathfrak{n}]$ is an Abelian ideal
in $\mathfrak{g}'=\mathfrak{g}/[\mathfrak{n},\mathfrak{n}]$,  its  intersection with $\mathfrak{h}'=(\mathfrak{h}+[\mathfrak{n},\mathfrak{n}])/[\mathfrak{n},\mathfrak{n}]$ is nonzero. On the other hand, we have the induced Levi decomposition $\mathfrak{g}'=\mathfrak{s}_c+\mathfrak{s}_{nc}
+(\mathfrak{r}/[\mathfrak{n},\mathfrak{n}])$. Then
we can apply
Lemma \ref{GO-Abelian-ideal-in-m-KVFCL} and a similar argument as for
(\ref{condition-1-in-key-lemma-1}) to prove that
$[\mathfrak{s}_{nc},\mathfrak{n}/[\mathfrak{n},\mathfrak{n}]]=0$, i.e.,
$[\mathfrak{s}_{nc},\mathfrak{n}]\subset[\mathfrak{n},\mathfrak{n}]$. This proves (\ref{condition-2-in-key-lemma-1}).
\end{proof}

When $M$ is simply connected, we have the following criterion for geodesic orbit Finsler nilmanifolds.

\begin{lemma}\label{key-lemma-2}
Keep  all the assumptions and notation in this section. If
$M$ is simply connected, and we have $\mathfrak{s}_{nc}=0$
and $\mathfrak{s}_c\subset\mathfrak{h}$,
then $(M,F)$ is a Finsler nilmanifold.
\end{lemma}

\begin{proof}
The radical $\mathfrak{r}$ coincides with the $B_\mathfrak{g}$-complement of $[\mathfrak{g},\mathfrak{g}]
=\mathfrak{s}_c+[\mathfrak{r},\mathfrak{g}]\subset\mathfrak{s}_c+\mathfrak{n}$. Since
$B_\mathfrak{g}(\mathfrak{n},\mathfrak{g})=0$, $\mathfrak{r}$ is also the $B_\mathfrak{g}$-complement of
$\mathfrak{s}_c$. By the assumption that $\mathfrak{s}_c\subset\mathfrak{h}$, we get $\mathfrak{n}\subset\mathfrak{m}\subset\mathfrak{r}$.

With respect to the $\mathrm{ad}(\mathfrak{h})$-action on $\mathfrak{m}$, we have the decomposition
\begin{equation}\label{101}
\mathfrak{m}=\mathfrak{c}_\mathfrak{m}(\mathfrak{h})+[\mathfrak{h},\mathfrak{m}]\subset
\mathfrak{c}_\mathfrak{m}(\mathfrak{h})+\mathfrak{n}.
\end{equation}
By Lemma \ref{lemma-go-centralizer}, the centralizer
$\mathfrak{c}_\mathfrak{m}(\mathfrak{h})$ of $\mathfrak{h}$ in $\mathfrak{m}$ is a compact subalgebra
in the solvable Lie algebra $\mathfrak{r}$, so it must be Abelian. Denote by $H_0$ the connected subgroup
generated by $\mathfrak{c}_\mathfrak{m}(\mathfrak{h})$. By Lemma \ref{lemma-go-centralizer},
we can write $G/H$ as $G'/H'=(G\times H_0)/HH_0$ in which the $H_0$-factor in the denominator is diagonally
imbedded in $G\times H_0$, and $F$ is $G\times H_0$-invariant.

By (\ref{101}), the nilpotent closed subgroup $N\times H_0$ in $G\times H_0$, where $N=\mathrm{Nil}(G)$,
acts transitively on $M=G/H$. For a generic closed subgroup $H'_0$ in $H_0$ with $\dim H'_0=\dim M-\dim N$,
$N\times H'_0$ acts transitively on $M$. Since $M$ is simply connected,
$N\times H'_0$ acts freely on $M$. The metric $F$ is $N\times H'_0$-invariant, so $(M,F)$
can be identified as the Finsler nilmanifold $(N\times H'_0,F)$.
\end{proof}

Now we further assume that $M$ is diffeomorphic to an Euclidean space. In this case  it is easily seen that
the isotropy subgroup $H$ is a maximal compact subgroup of $G$.

\begin{lemma}
Assume $(G/H,F)$ is a homogeneous Finsler space diffeomorphic to an Euclidean space, then $H$ is
a maximal compact subgroup of $G$.
\end{lemma}

\begin{proof}
Assume conversely that $H$ is not maximal. Then there is  a maximal compact subgroup $K$ of $G$ such that $H\subset K$. By the
first manifold splitting theorem (Theorem 14.3.8 in \cite{HN2012}), $G/K$ is also
diffeomorphic to a Euclidean space, and $K$ is connected.
Then the assumption that $H$ is not maximal implies that $\dim K-\dim H=r>0$.
Since $G/H$ is the total space of a $K/H$-bundle over $G/K$,
$G/H$ is homotopic to $K/H$. So we have
$$H_r(G/H;\mathbb{Z}_2)=H_r(K/H;\mathbb{Z}_2)\neq 0,$$
which is a contradiction.
\end{proof}

All maximal compact subgroups of $G$ are conjugate to each other, and for any maximal
subgroup of $G$, we can find a maximal compact subgroup containing it. So with respect to the
fixed Levi decomposition, we can choose a suitable $H$, such that $S_c K'_{nc}\subset H$. Meanwhile
$HR/R$ is a compact subgroup of $S$, we may further assume $H\subset S_c K_{nc} R$, i.e., in the
Lie algebra level, we have
$\mathfrak{s}_c+\mathfrak{k}'_{nc}\subset\mathfrak{h}
\subset\mathfrak{s}_c+\mathfrak{k}_{nc}$.
Denote by $\mathfrak{h}_{nc}$ the projection of $\mathfrak{h}$ in
$\mathfrak{s}_{nc}=\mathfrak{g}/(\mathfrak{s}_c+\mathfrak{r})$. Then
we have
\begin{eqnarray*}
\mathfrak{k}'_{nc}\subset\mathfrak{h}_{nc}\subset\mathfrak{k}_{nc}, \mbox{ and }
\mathfrak{h}+\mathfrak{r}=\mathfrak{s}_c+\mathfrak{h}_{nc}+\mathfrak{r}.
\end{eqnarray*}

In the case that  $G=I_0(M,F)$, we have the following lemma.

\begin{lemma}\label{key-lemma-3}
Let $(M,F)=(G/H,F)$  be a geodesic orbit Finsler space diffeomorphic
to a Euclidean space, where  $G=I_0(M,F)$. Keep all the notation and assumptions. Then $\mathfrak{h}_{nc}=\mathfrak{k}_{nc}$.
\end{lemma}

\begin{proof}
Assume conversely that $\mathfrak{h}_{nc}\neq\mathfrak{k}_{nc}$. Then
there exists a nonzero Abelian subalgebra
$\mathfrak{h}_0$, which is
the orthogonal complement of $\mathfrak{h}_{nc}$ in $\mathfrak{k}_{nc}$
with respect to the Killing form $B_{\mathfrak{s}_{nc}}$ of
$\mathfrak{s}_{nc}$. The subalgebra $\mathfrak{h}_0$ generates a closed Abelian subgroup $H_0$ in $S_{nc}$.

By Lemma \ref{key-lemma-1}, we have $\mathfrak{h}_0\subset \mathfrak{c}_\mathfrak{m}(\mathfrak{h})$.
By Lemma \ref{lemma-go-centralizer},
we can also write $M$ as $G\times H_0/HH_0$ such that
$F$ is $G\times H_0$-invariant.

Now the assumption  $G=I_0(M,F)$ implies that, for any nonzero
$v\in\mathfrak{h}_0$, there exists a vector $u\in\mathfrak{h}$
such that $[u-v,\mathfrak{m}]=0$. Notice that $[v,\mathfrak{m}]=[v,\mathfrak{m}]_\mathfrak{m}\subset\mathfrak{m}$, Since $[\mathfrak{c}_\mathfrak{m}(\mathfrak{h}),\mathfrak{m}]
\subset\mathfrak{m}$. Obviously $u-v\notin\mathfrak{c}(\mathfrak{g})$, Since in the Levi decomposition, the $\mathfrak{s}_{nc}$-summand of $u-v$ is nonzero.
So $[u-v,\mathfrak{h}]=[u,\mathfrak{h}]\neq0$ generates
a nonzero ideal
$$[u,\mathfrak{h}]+[[u,\mathfrak{h}],\mathfrak{h}]+
[[[u,\mathfrak{h}],\mathfrak{h}],\mathfrak{h}]+\cdots$$
of $\mathfrak{g}$ contained in $\mathfrak{h}$.
This is a contradiction to the effectiveness of the $G$-action on $M$.
\end{proof}

\section{Proof of Theorem \ref{main-thm-1} and further discussion}

We keep all the notation and assumptions in the previous section.

{\bf Proof of Theorem \ref{main-thm-1}.}
Write the geodesic orbit Finsler space $(M,F)$ as $(G/H,F)$, where
$G=I_0(M,F)$. We
First construct the submersion. Applying Theorem \ref{go-submersion-thm} to
the closed normal subgroup $S_cR$ of $G$, we get a submersion
$\pi:(G/H,F)\rightarrow(G/S_cRH,F')$, where $F'$ is also a $G$-geodesic orbit
Finsler metric. If we write $G/S_cRH$ as $G'/H'$ such that $G'$ is connected
and acts effectively, then $\mathfrak{g}'=\mathrm{Lie}(G')$ can be identified with
$\mathfrak{s}_{nc}$. By Lemma \ref{key-lemma-3}, $\mathfrak{h}'=\mathrm{Lie}(H')
\subset\mathfrak{s}_{nc}$ can then be identified with $\mathfrak{k}_{nc}$, the
maximal compact subalgebra of $\mathfrak{s}_{nc}$.
The pair $(\mathfrak{s}_{nc},\mathfrak{k}_{nc})$ is a symmetric pair, so $(G'/H',F')$
is a symmetric Finsler space. Notice that $G'/H'$ must be simply connected,
Since each irreducible factor of
$(\mathfrak{s}_{nc},\mathfrak{k}_{nc})$ is of non-compact type. It is well known that any homogeneous Finsler
metric on $G'/H'$ must be symmetric, and thus Berwaldian.

Next we consider the fibers of $\pi$. Since $\pi$ defines a fiber bundle for which
both the base and the total space are diffeomorphic to Euclidean spaces, the fibers must be simply connected.

Now we prove that the fiber of $\pi$ containing $o=eH\in M$ is totally geodesic.
By the construction of the submersion in Theorem \ref{go-submersion-thm},
the fiber coincides with the orbit $R\cdot o$, and
$\dim R\cdot o=\dim R-\dim H\cap R$. By Lemma \ref{key-lemma-3},
$\dim H\cap RK''_{nc}=\dim H\cap R+\dim K''_{nc}$,
so we have
\begin{eqnarray*}
\dim R\cdot o&=&\dim R-\dim H\cap RK''_{nc}+\dim K''_{nc}\\
&=&(\dim S_c+\dim K''_{nc}+\dim R)-(\dim S_c+\dim H\cap RK''_{nc})\\
&=&\dim S_cK''_{nc}R-\dim H\cap S_cK''_{nc}R\\
&=&\dim S_cK''_{nc}R\cdot o.
\end{eqnarray*}
Since $R\cdot o\subset S_cK''_{nc}R\cdot o$, the fiber of $\pi$ containing $o$
can also be identified with the orbit $S_cK''_{nc}R\cdot o$. The group $S_cK''_{nc}R$
is the identity component of the centralizer of $K_{nc}$ in $G$, so
$S_cK''_{nc}R\cdot o$ is a connected totally geodesic submanifold.

Finally, we prove that the fiber of $\pi$ containing $o$ is a geodesic orbit nilmanifold,
with  step-size at most 2.
By Lemma \ref{totally-geodesic-technique-lemma-for-go}, the restriction of $F$ to
that fiber is an $S_cK''_{nc}R$-geodesic orbit metric. The Levi subgroup
of $S_cK''_{nc}R$ is compact, and $R\cdot o$ is simply connected,
so by Lemma \ref{key-lemma-2}, the fiber of $\pi$
containing $o$ is a nilmanifold.
By Theorem \ref{2-step-nil-go-thm}, its step-size
is at most 2.

By the homogeneity, other fibers of $\pi$ share the same properties.
This ends the proof of Theorem \ref{main-thm-1}.\
\rule{0.5em}{0.5em}

Next we discuss some special cases.

It is an interesting problem to explore when the fiber bundle description \ref{main-thm-1} can be strengthened to a product
(in the sense of homogeneous space rather than metric) of a nilmanifold and a symmetric space of non-compact type.

Notice that the product structure can not always be achieved for $(M,F)$ in Theorem \ref{main-thm-1}.
The following example was found
by C. Gordon. Let $M$ be the universal covering of $SL(2,\mathbb{R})$. Then $M$ is diffeomorphic to $\mathbb{R}^3$. It admits a left invariant geodesic orbit
Riemannian metric $F$, for which $\mathfrak{g}=\mathrm{Lie}(I_0(M,F))$ is isomorphic to
$sl(2,\mathbb{R})\oplus\mathbb{R}$ \cite{Go1996}. The Lie algebra $\mathfrak{h}=\mathbb{R}$ of the isotropy group $H$ is
diagonally imbedded in $\mathfrak{g}$. So the coset space $M$ can not be identified with
$\mathbb{R}\times (SL(2,\mathbb{R})/SO(2))$, though topologically it is.

However, the above example implies
that the trouble is caused by
the Abelian factor $\mathfrak{k}''_{nc}$ in $\mathfrak{s}_{nc}$.
When $\mathfrak{k}''_{nc}=0$, i.e.,  the symmetric space $S_{nc}/K_{nc}$ does not contain any Hermitian factor, we can prove the following

\begin{theorem}\label{main-thm-2}
Let $(M,F)$ be a geodesic orbit Finsler space diffeomorphic to
a Euclidean such that in the Levi decomposition $\mathfrak{g}=
\mathfrak{s}+\mathfrak{r}=\mathfrak{s}_c+\mathfrak{s}_{nc}+
\mathfrak{r}$ for $G=I_0(M,F)$, the maximal compact subalgebra
$\mathfrak{k}_{nc}$ in $\mathfrak{s}_{nc}$ is semi-simple, then
 $M$ can be written as $M=G/H$ such that
\begin{description}
\item{\rm (1)}
$G=G_1\times G_2$, $H=H_1\times H_2$
and $M=M_1\times M_2=G_1/H_1\times G_2/H_2$, where $G_1=S_{nc}$,
$G_2=S_cR$, $H_1=K_{nc}$ and $H_2=H\cap RS_c$.
\item{\rm (2)}
For any $x_1\in M_1$ and $x_2\in M_2$, $M_1\times x_2$ and $x_1\times M_2$
are totally geodesic. Further more, $(M_1\times x_2,F|_{M_1\times x_2})$ is a $G_1$-geodesic orbit space,
and $(x_1\times M_2,F|_{x_1\times M_2})$ is a $G_2$-geodesic orbit space.
\end{description}
\end{theorem}

\begin{proof}
Since $\mathfrak{k}_{nc}$ is semi-simple, it generates a maximal compact subgroup $K_{nc}$
in $S_{nc}$. We can choose the suitable isotropy subgroup $H$,
which is a maximal compact subgroup of $G$, such that
$S_cK_{nc}\subset H$. On the other hand, the projection of $H$
in $G/R$ is a compact subgroup of $S$ containing $S_cK_{nc}$,
where the equality must happen since $S_cK_{nc}$ is a maximal
compact subgroup of $S$.

By Lemma \ref{key-lemma-1}, we have the Lie algebra direct sum decomposition
$$\mathfrak{g}=\mathfrak{g}_1\oplus\mathfrak{g}_2
=\mathfrak{s}_{nc}\oplus(\mathfrak{s}_c+\mathfrak{r}).$$
By the above observation, we have  a Lie algebra direct sum decomposition for $\mathfrak{h}=\mathrm{Lie}(H)$,
$$\mathfrak{h}=\mathfrak{h}_1\oplus\mathfrak{h}_2
=\mathfrak{k}_{nc}\oplus(\mathfrak{s}_c+\mathfrak{r}\cap\mathfrak{h}).
$$
The subgroups $G_1$, $G_2$, $H_1$ and $H_2$ are connected
subgroups generated by $\mathfrak{g}_1$, $\mathfrak{g}_2$,
$\mathfrak{h}_1$ and $\mathfrak{h}_2$, respectively, where
$G_1$ and $G_2$ are closed, and $H_1$ and $H_2$ are compact.

Obviously $G_1$ and $G_2$ commute with each other. To prove $G=G_1\times G_2$, we only need to prove $G_1\cap G_2=\{e\}$.
Let $g$ be any element in $G_1\cap G_2$.
Given $x\in M$,  restricted to the orbit $(G_1\cdot x,F|_{G_1\times x})$, the $g$-action defines a Clifford-Wolf translation.
The homogeneous Finsler space $(G_1\cdot x,F|_{G_1\times x})$
is a symmetric  Finsler space of non-compact type,
which has non-positive flag curvature and negative Ricci scalar, so it does not admit any non-trivial Clifford-Wolf translation
\cite{DH2007}. So $g$ acts trivially on each $G_1\times x$, i.e.,
$g$ acts trivially on $M$. By the effectiveness of the $G$-action,
we must have $g=e$. Thus $G=G_1\times G_2$, and
 $H=H_1\times H_2$. Then $M=M_1\times M_2=(G_1/H_1)\times (G_2/H_2)$ follow immediately. This proves (1) in the theorem.

By Theorem \ref{main-thm-1},
each $(x_1\times M_2,F|_{x_1\times M_2})$ is totally geodesic in
$(M,F)$.  By Lemma \ref{totally-geodesic-technique-lemma-for-go}, it is a $G_2$-geodesic orbit space itself. Now we prove this statement for
$M_1\times x_2$.

Let $\mathfrak{m}$ be the $B_\mathfrak{g}$-complement of
$\mathfrak{h}$. Then $\mathfrak{m}=\mathfrak{m}_1+\mathfrak{m}_2$,
 where $\mathfrak{m}_i$ is the $B_{\mathfrak{g}}$-complement (which is the same as the $B_{\mathfrak{g}_i}$-complement) of
$\mathfrak{h}_i$ in $\mathfrak{g}_i$ respectively. Then
we have the reductive decomposition
$\mathfrak{g}_i=\mathfrak{h}_i+\mathfrak{m}_i$, and
$\mathfrak{m}_i$ can be identified with the tangent space of
$M_i=G_i/H_i$ at $o_i=eH_i$.

By  Proposition \ref{prop-1},  for any non-zero $u\in\mathfrak{m}=\mathfrak{m}_1+\mathfrak{m}_2$,
there exists a vector $v=v_1+v_2\in\mathfrak{h}=\mathfrak{h}_1\oplus\mathfrak{h}_2$,
such that $\eta(u)=[u,v]$. In particular, if $u\in\mathfrak{m}_1$, then we have
\begin{eqnarray*}
& &\eta(u)=[u,v_1+v_2]=[u,v_1]\in\mathfrak{m}_1,\mbox{ and }\\
& &\langle\eta(u),v\rangle_u^F=\langle u,[v,u]_{\mathfrak{m}_1}\rangle_u^F,\quad\forall v\in\mathfrak{m}_1.
\end{eqnarray*}
 Now we make two observations based on the above argument. First, the restriction of the spray vector field $\eta(\cdot)$ to $\mathfrak{m}_1$ coincides with the spray
vector field of $(G_1/H_1\times o_2,F|_{G_1/H_1\times o_2})$.
Thus by Lemma \ref{criterion for totally geodesic homogeneous subspace}, the Finsler submanifold
$(G_1/H_1\times o_2,F|_{G_1/H_1\times o_2})$ is totally geodesic in $(M,F)$. Second, the spray
vector field of $(G_1/H_1,F|_{G_1/H_1\times o_2})$, which coincides with the restriction of $\eta(\cdot)$ to $\mathfrak{m}_1$, is tangent to $\mathrm{Ad}(H_1)$-orbits in $\mathfrak{m}_1$. Then
by Proposition \ref{prop-1},
$(G_1/H_1\times o_2,F|_{G_1/H_1\times o_2})$ is a $G_1$-geodesic orbit Finsler space. By the homogeneity, each $(G_1/H_1\times x_2,F|_{G_1/H_1\times x_2})$ is totally geodesic in $(M,F)$
and it is a $G_1$-geodesic orbit space itself.

This proves (2), completing the proof of the theorem.
\end{proof}

Finally, we consider the interaction between the geodesic orbit condition and the negative (non-positive) curvature condition.

In \cite{XD2017}, we stated following theorems.

\begin{theorem}\label{main-cor-1}
Any negatively curved geodesic orbit Finsler space $(M,F)$ is a
rank-one Riemannian symmetric space of non-compact type.
\end{theorem}

\begin{theorem}\label{main-cor-2}
For any geodesic orbit Finsler space $(M,F)$ with non-positive flag curvature
and negative Ricci scalar, $M$ is a symmetric  space of non-compact type
and $F$ is Berwaldian.
\end{theorem}

Notice that any homogeneous Finsler space with negative flag curvature, or with non-positive
flag curvature and negative Ricci scalar, is simply connected
\cite{DH2007}, and thus diffeomorphic to
an Euclidean space \cite{BCS}.

Here we give two alternative proofs for the above theorems using the techniques in the previous sections.

{\bf Proof of Theorem \ref{main-cor-1}.}
Denote $G=I_0(M,F)$, $\mathfrak{g}=\mathfrak{s}+\mathfrak{r}=
\mathfrak{s}_c+\mathfrak{s}_{nc}+\mathfrak{r}$ the Levi decomposition for
$\mathfrak{g}=\mathrm{Lie}(G)$. Write $M$ as $M=G/H$, with a $B_\mathfrak{g}$-orthogonal
reductive decomposition $\mathfrak{g}=\mathfrak{h}+\mathfrak{m}$.

 We first prove $\mathfrak{r}=0$. Assume conversely that $\mathfrak{r}\neq 0$. Then
there exists  a nonzero Abelian ideal $\mathfrak{a}$ of $\mathfrak{g}$ contained in $\mathfrak{n}=\mathfrak{nil}(\mathfrak{g})$.
So the intersection of $\mathfrak{a}\subset\mathfrak{m}$  with $\mathfrak{h}$ is nonzero.
By Lemma \ref{GO-Abelian-ideal-in-m-KVFCL}, each nonzero vector $v\in \mathfrak{a}$ defines
a nonzero Killing vector field $V$ of constant length. By Theorem 5.1 in \cite{XD2017-2}, the flag curvature $K(x,y,\mathbf{P})$
is non-negative whenever $y=V(x)$ for some $x\in M$. This is a contradiction with the negatively curved condition. Thus $\mathfrak{r}=0$.

 Recall that  any negatively curved homogeneous Finsler space must be simply connected \cite{DH2007}. Thus $M$ must be
diffeomorphic to an Euclidean space. We may choose a suitable isotropy subgroup $H$, such that
$\mathfrak{s}_c+\mathfrak{k}'_{nc}\subset\mathfrak{h}\subset\mathfrak{s}_c+\mathfrak{k}_{nc}$.
By the effectiveness of the $G$-action, we must have $\mathfrak{s}_c=0$. By Lemma \ref{key-lemma-3},
$\mathfrak{h}=\mathfrak{h}_{nc}=\mathfrak{k}_{nc}$. So $M=S_{nc}/K_{nc}$ is a symmetric
space of non-compact type. The invariant metric $F$ is symmetric, and thus
Berwaldian. So it is negatively curved if and only if $G/H$ is of rank one \cite{De2012}. In this case,
the isotropy subgroup $H$ acts transitively on the $F$-unit sphere in $T_oM$. Therefore the homogeneous Finsler
metric $F$ must be a Riemannian symmetric metric, which is uniquely determined up to a scalar.

This completes the proof of Theorem \ref{main-cor-1}.\ \rule{0.5em}{0.5em}

The proof of Theorem \ref{main-cor-2} is similar, and can be omitted.

\end{document}